\documentclass{amsart}

\usepackage[english]{babel}
\usepackage{mathrsfs}

\newtheorem{theorem}{Theorem}[section]
\newtheorem{corollary}[theorem]{Corollary}
\newtheorem{lemma}[theorem]{Lemma}
\newtheorem{proposition}[theorem]{Proposition}
\theoremstyle{definition}
\newtheorem{example}[theorem]{Example}
\newtheorem{remark}[theorem]{Remark}

\numberwithin{equation}{section}

\newcommand*{\cf}{C_{\phi}}
\newcommand*{\hh}{\mathcal{H}}
\newcommand*{\h}{\mathsf{h}}

\newcommand*{\B}[1][\hh]{\boldsymbol{B}(#1)}
\newcommand*{\ascr}{\mathscr{A}}
\newcommand*{\bscr}{\mathscr{B}}

\newcommand*\nbb{\mathbb{N}}
\newcommand*\rbb{\mathbb{R}}
\newcommand*\C{\mathbb{C}}
\newcommand*\zbb{\mathbb{Z}}

\newcommand*\Rng{\mathcal{R}}
\newcommand*\Ker{\mathcal{N}}
\newcommand*{\pp}[1]{^{[#1]}}

\def\RngInf{\operatorname{\Rng^\infty}}

\newcommand{\ushift}[1]{#1 \sim \{#1_n\}_{n=0}^\infty}

\begin{document}

\title[The Wold-type decomposition for $m$-isometries]{The Wold-type decomposition for $m$-isometries}
\author{Jakub Ko\'smider}
\address{Instytut Matematyki, Uniwersytet Jagiello\'nski, ul. \L{}ojasiewicza 6, PL-30348 Kra\-k\'ow, Poland}
\email{jakub.kosmider@im.uj.edu.pl}
\thanks{}
\dedicatory{}


\begin{abstract}
The aim of this paper is to study the Wold-type decomposition in the class of $m$-isometries.
One of our main results establishes an equivalent condition for an analytic $m$-isometry to admit the Wold-type decomposition for $m\ge2$. 
In particular, we introduce the $k$-kernel condition which we use to characterize analytic $m$-isometric operators which are unitarily equivalent to unilateral operator valued weighted shifts for $m\ge2$. 
As a result, we also show that $m$-isometric composition operators on directed graphs with one circuit containing only one element are not unitarily equivalent to unilateral weighted shifts.
We also provide a characterization of $m$-isometric unilateral operator valued weighted shifts with positive and commuting weights.
\end{abstract}

\maketitle

\section{Introduction}

In 1961 Halmos published a paper concerning shift operators on Hilbert spaces (see \cite{halmos-1961}).
The paper introduced the notion of the wandering subspace and its connections with invariant subspaces of unilateral and bilateral shifts. 
This notion is closely related to the famous Wold decomposition of isometries which allows to represent an isometry as orthogonal sum of a unitary operator and a unilateral weighted shift (cf. \cite[Chapter I, Theorem 1.1]{N-F-2010}).
Later on, Shimorin introduced a more general property named the Wold-type decomposition which generalizes the classical Wold decomposition and proved in \cite[Theorem~3.6]{shimorin-2001} that it applies to a broader class of operators than isometries.
To prove some of the results, Shimorin utilized a concept from the paper of Richter from 1988, who showed that an analytic $2$-concave operator has the wandering subspace property.
The idea was based on a specific operator series employing a particular left-inverse of an operator  (see \cite[Theorem 1]{richter-1988}).
In 2019 Anand et al.\ characterized non-unitary $2$-isometries which admit the Wold-type decomposition having their non-unitary parts unitarily equivalent to unilateral operator valued weighted shifts (see \cite[Theorem 2.5]{A-C-J-S-2017}).
Their description of these operators exhibits the form of the weights and allows one to construct interesting non-unitary $2$-isometries.
This result shows a greater resemblance of the classical Wold decomposition.

The aforementioned notions are closely related to each other as if an operator admits the Wold-type decomposition, then its non-unitary part has the wandering subspace property.
In recent years authors published interesting results concerning the Wold-type decomposition and the wandering subspace property in the $m$-isometric context. 
Some of them focus on finding analytic $m$-isometries without the wandering subspace property (see \cite{A-C-T-2018}), whereas others focus on more general results (see \cite{A-C-J-S-2017}).

The aim of this paper is to adopt the Wold-type decomposition in the class of $m$-isometries for $m\ge2$. 
We focus on finding conditions under which these operators admit the aforementioned decomposition. 
In order to do so we, in particular, introduce the $k$-kernel condition in the $m$-isometric context and demonstrate its connection with the Wold-type decomposition.  

The paper is organized as follows. 
Section 2 is a short introduction to the $k$-kernel condition and aggregates facts concerning its properties for bounded operators. 
In the beginning, we discuss the properties of $m$-isometric operators satisfying the $(m-1)$-kernel condition for $m\ge2$.
As a result, in Proposition \ref{pro:analytic-inclusions} we obtain an interesting property regarding inclusions of certain images of the kernel of the adjoint under powers of an operator, which is vital in the sequel.
We also establish a characterization of the $k$-kernel condition in the class of left-invertible composition operators (see Theorem \ref{thm:composition-kernel-condition}).

In Section 3 we introduce the completion problem for $m$-isometric unilateral operator valued weighted shifts with positive and commuting weights. 
Our approach based on the spectral theory enables us to establish a handy characterization of $m$-isometric unilateral operator valued weighted shifts (see Theorem \ref{thm:unilateral-m-iso}).
This characterization is used later in the proof of Corollary \ref{cor:completion-problem} to provide a solution the $m$-isometric completion problem in the case $m-1$ initial weights are given.

Finally, Section 4 is devoted to the Wold-type decomposition for $m$-isometric operators for $m\ge2$.
Theorem \ref{thm:analytic-m-iso} is the main result of this section and provides an equivalent condition for analytic $m$-isometry to satisfy the $(m-1)$-kernel condition. 
This fact exhibits an interesting interplay between the wandering subspace property and $(m-1)$-kernel condition in the class of $m$-isometries.
Thanks to this development we obtain the positive answer to \cite[Question 1.2]{A-C-T-2018} in some subclass of operators (see Corollary~\ref{cor:wold-type-answer}).
Furthermore, we show in Theorem \ref{thm:analytic-m-iso-shifts} that $m$-isometries satisfying $(m-1)$-kernel condition are unitarily equivalent to unilateral operator valued weighted shifts.
We conclude this section with Corollary \ref{cor:composition-unitary-equivalence} that proves that the class of analytic $m$-isometric composition operators on a directed graph with one circuit having exactly one element is disjoint from the class of unilateral operator valued weighted shifts.

\section{Properties of the $k$-kernel condition}

In what follows, by $\nbb$, $\zbb_+$, $\rbb_+$ we denote the set of positive integers, non-negative integers and non-negative real numbers, respectively.
Throughout the paper $\hh$ denotes a nonzero complex Hilbert space.
If $x$, $y\in\hh$, then $\langle x,y\rangle$ is their inner product.
If $M$ is a closed linear subspace of $\hh$, then $P_M$ is the corresponding orthogonal projection of $\hh$ onto $M$ and $M^\perp$ is the orthogonal complement of $M$ in $\hh$.
For a subset $A\subseteq\hh$ by $\bigvee A$ we understand the closed linear span of $A$, that is, the smallest closed subspace containing $A$.
If $\{M_n\}_{n=0}^\infty\subseteq\hh$ is a sequence of subspaces, then we define
$\bigvee\{M_n\}_{n=0}^\infty=\bigvee\Big(\bigcup_{n=0}^\infty M_n\Big)$.

As usual $\B$ stands for the algebra of all bounded linear operators acting on $\hh$ and $I\in\B$ is the identity operator.
Suppose that $T\in\B$.
By $\Ker(T)$, $\Rng(T)$, and $T^*$ we denote the kernel, the range and the adjoint of $T$, respectively.
If there exists $S\in\B$ such that $ST=I$, then we say the $T$ is \textit{left-invertible}. 
It is surely folklore that left-invertible operators have closed ranges and any power of a left-invertible operator is again left-invertible.
For $k\in\nbb$ we denote the operator $T^{*k}T^k$ by $T\pp{k}$.
Recall that an operator $T\in\B$ is said to be \textit{analytic}, if $\RngInf(T):= \bigcap_{n=0}^\infty T^n(\hh)$ is a null space. 

\begin{remark}
	\label{rem:invertiblity}
	Observe that, if $T\in\B$ is left-invertible, then for any $k\in\nbb$ the operator $T\pp{k}$ is invertible.
	This follows from the fact that for any left-invertible operator $S\in\B$ the operator $S\pp{1}$ is invertible and the following equality
	\begin{align*}
		T\pp{k} = T^{*k}T^k=(T^k)^*T^k = (T^k)\pp{1}, \quad k\in\nbb.
	\end{align*}
\end{remark}

Assume $m\in\nbb$. 
Operator $T\in\B$ is said to be an \textit{$m$-isometry} if 
\begin{align}
	\label{eq:m-iso-condition}
	\sum_{p=0}^m (-1)^p \binom{m}{p} T\pp{p} = 0.
\end{align}
Observe that we can equivalently define $T\in\B$ to be an $m$-isometry if and only if for every $k\ge m$ it follows that
\begin{align}
	\label{eq:m-iso-inverse}
	T\pp{k} = (-1)^{m+1} \sum_{p=0}^{m-1} (-1)^p \binom{m}{p} T\pp{k-m+p}.
\end{align}
It is well known that $m$-isometries are left-invertible.
If $m\ge2$, then an $m$-isometry is said to be \textit{strict}, if it is not an $(m-1)$-isometry.
For more introductory information regarding $m$-isometric operators and their properties the reader is referred to \cite{A-S-1995}.

We say that $T\in\B$ satisfies the \textit{kernel condition} if $\Ker(T^*)$ is invariant for $T\pp{1}$.
This notion was introduced in \cite{A-C-J-S-2017}  and studied in the context of $2$-isometries.
We extend it to the $k$-kernel condition, namely, $T$ satisfies the $k$-\textit{kernel condition} for some $k\in\nbb$, if $\Ker(T^*)$ is invariant for $T\pp{n}$ for all $n\in\{1,\dots,k\}$.
It is an easy observation that the kernel condition coincides with the $1$-kernel condition.
Furthermore, if $T\in\B$ satisfies the $k$-kernel condition for some $k\ge2$, then it also satisfies the $(k-1)$-kernel condition.
Finally, observe that $T\in\B$ satisfies the $k$-kernel condition if and only if $\Ker(T^*)$ is reducing for  $T\pp{n}$ for all $n\in\{1,\dots,k\}$

It is a direct consequence of the kernel-range decomposition that in the class of invertible operators all of the above conditions are trivially satisfied.
Also it is easy to prove that isometries and unilateral weighted shifts satisfy the $k$-kernel condition for every $k\in\nbb$.
On the other hand, the situation is more interesting in the class of analytic operators as the kernel of the adjoint of an analytic operator may not be a null space.
Hence, it is worth searching for examples there.

The first fact we prove shows that in the class $m$-isometries the $(m-1)$-kernel condition automatically implies the $k$-kernel condition for every $k\in\nbb$.

\begin{lemma}
	\label{lem:kernel-condition}
	Suppose that $T\in\B$ is an $m$-isometry for $m\ge2$.
	Then $T$ satisfies the $(m-1)$-kernel condition if and only if $T$ satisfies the $k$-kernel condition for every $k\in\nbb$.
\end{lemma}
\begin{proof}
	Assume that $T$ satisfies the $(m-1)$-kernel condition and let $k\ge m$.
	We use an induction argument. 
	Suppose that $T$ satisfies the $(k-1)$-kernel condition.
	To finish the proof it is enough to note that, if $f\in\Ker(T^*)$, then
	\begin{align*}
		T\pp{k} f  \overset{\eqref{eq:m-iso-inverse}} = (-1)^{m+1} \sum_{p=0}^{m-1} (-1)^p \binom{m}{p} T\pp{k - m+p} f \in \Ker(T^*),
	\end{align*}
	since $k-m+p\le k-1$ for $p\in\{0,\dots, m-1\}$.
\end{proof}

The following result is a consequence of the above lemma and provides an interesting inclusion between certain subspaces for $m$-isometric operators satisfying the $(m-1)$-kernel condition.

\begin{proposition}
	\label{pro:analytic-inclusions}
	Suppose $m\ge 2$ and $T\in\B$ is an $m$-isometry satisfying the $(m-1)$-kernel condition.
	Then $T^{n-1}(\Ker(T^*))\subseteq\Ker(T^{*n})$ for every $n\in\nbb$.
\end{proposition}
\begin{proof}
	Assume that $n\in\nbb$ and $f\in T^{n-1}(\Ker(T^*))$. 
	There exists $h\in\Ker(T^*)$ such that $f=T^{n-1}h$.
	Consequently, 
	\begin{align*}
		T^{*n}f = T^{*n} T^{n-1} h = T^*(T\pp{n-1}) h = 0,
	\end{align*}
	as $T\pp{n-1} h\in\Ker(T^*)$ due to the fact that Lemma \ref{lem:kernel-condition} implies that $T$ satisfies the $(n-1)$-kernel condition.
	This completes the proof.
\end{proof}

Let us recall that, if $M\subseteq\hh$ is a closed subspace and $T\in\B$, then by $T_{|M}$ we denote the restriction of $T$ to the subspace $M$.
The technical result below shows that the $k$-kernel condition behaves very well under taking restriction to a reducing subspace.
We shall use it in the sequel.

\begin{lemma}
	\label{lem:kernel-condition-reducing}
	Suppose that $T\in\B$ satisfies the $k$-kernel condition for some $k\in\nbb$ and $M$ is a non-trivial reducing subspace for $T$.
	Then $T_{|M}$ satisfies the $k$-kernel condition.
\end{lemma}
\begin{proof}
	Let $n\in\{1,\dots,k\}$.
	Set $S=T_{|M}$ and let $f\in\Ker(S^*)$.
	It is surely folklore that $(T^*)_{|M} = (T_{|M})^*$ and $\Ker(S^*) = M\cap \Ker(T^*)$.
	Furthermore, using the above we can show that
	\begin{align*}
		S\pp{n}f = S^{*n}S^n f= (T_{|M})^{*n} (T_{|M})^n f = T\pp{n}f \in \Ker(T^*) \cap M  = \Ker(S^*).
	\end{align*} 
	Hence the proof is completed.
\end{proof}

It turns out the $k$-kernel condition is an invariant of unitary equivalence.
The proof relies on the fact that for $T$, $S\in\B$ such that $S=UTU^*$ for some unitary $U\in\B$ it is true that $S\pp{n}=UT\pp{n}U^*$ for any $n\in\nbb$ and $U\Ker(T^*)=\Ker(S^*)$.
However, for the reader's convenience we present a direct proof of this fact.

\begin{proposition}
	Suppose $T\in\B$ and $S=UTU^*$ for some unitary operator $U\in\B$. 
	If $T$ satisfies the $k$-kernel condition for some $k\in\nbb$, then so does $S$.
\end{proposition}
\begin{proof}
	Assume that $T$ satisfies the $k$-kernel condition for some $k\in\nbb$ and 
	observe that $\Ker(S^*) = \Rng(S)^\perp = (U\Rng(T))^\perp$. 
	Thus
	\begin{align}
		\label{eq:kc-scalar-product}
		0 = \langle Uh, f\rangle, \quad f\in\Ker(S^*), h\in\overline{\Rng(T)},
	\end{align}
	as the inner product is continuous.
	To prove that $S$ satisfies the $k$-kernel condition assume that $f\in\Ker(S^*)$.
	It is easily seen that $S\pp{n}=UT\pp{n}U^*$ for $n\in\zbb_+$. 
	Observe that, since $T$ satisfies the $k$-kernel condition, the kernel-range decomposition implies that $T\pp{n} g\in \overline{\Rng(T)}$ for every $g\in\Rng(T)$ and every $n\in\{1,\dots, k\}$.
	Hence
	\begin{align*}
		\langle Ug, S\pp{n}f\rangle = \langle g, T\pp{n}U^*f\rangle = \langle UT\pp{n} g, f\rangle 
		\overset{\eqref{eq:kc-scalar-product}}= 0, \quad g\in\Rng(T), n\in\{1,\dots, k\}.
	\end{align*}
	This implies that $S\pp{n}f\in\Ker(S^*)$ for $n\in\{1,\dots, k\}$ and completes the proof.
\end{proof}

Before presenting the next result we recall necessary definitions and notation related to composition operators adopted from \cite{nordgren-1978, J-K-2019}.

Let $(X,\ascr,\mu)$ be a measure space such that $\mu$ is non-negative and $\sigma$-finite.
If $X$ is countably infinite, all singletons are measurable and $\mu(\{x\})\in(0,\infty)$ for every $x\in X$, then we say that $(X,\ascr,\mu)$ is a  \textit{discrete measure space}.
If $(X,\ascr,\mu)$ is a discrete measure space and $x\in X$, then we define $\mu(x):=\mu(\{x\})$. 

From now on denote by $\phi\colon X\to X$ a measurable mapping that is also \textit{nonsingular}, i.e.\ $\mu\circ\phi^{-1}$ is absolutely continuous with respect to $\mu$. 
Operator $\cf\in\B[L^2(\mu)]$ is called a \textit{composition operator}, if $\cf f:=f\circ\phi$ for every $f\in L^2(\mu)$, where $ L^2(\mu)=L^2(X,\ascr,\mu)$.
Due to nonsingularity of $\phi$, it follows from the Radon-Nikodym theorem that there exists a unique measurable function $\h\colon X\to[0,\infty]$ such that
\begin{align*}
	\mu\circ\phi^{-1}(\Delta) = \int_{\Delta} \h \mathrm{d}\mu, \quad \Delta\in\ascr.
\end{align*}
It is well known that the function $\h$ plays a very important role in theory of composition operators as it allows one to describe many properties of the associated composition operator using properties of $\h$ (see, e.g., \cite{nordgren-1978,J-K-2019,singh-1974}).
For more introductory information regarding composition operators the reader is referred to \cite{nordgren-1978}.

We conclude this section with a generalization of \cite[Proposition~4.8]{J-K-2019} that establishes a characterization of left-invertible bounded composition operators satisfying the $k$-kernel condition for $k\in\nbb$.

\begin{theorem}
	\label{thm:composition-kernel-condition}
	Assume $(X,\ascr,\mu)$ is a discrete measure space, $\phi\colon X\to X$ is a function, $\cf\in\B[L^2(\mu)]$ is left-invertible and $k\in\nbb$.
	Then $\cf$ satisfies the $k$-kernel condition if and only if $\h$ is constant on preimages $\phi^{-i}(\{x\})$ for all $i\in\{1,\dots,k\}$ and for every $x\in X$.
\end{theorem}
\begin{proof}
	First, let us gather the following facts which are required in the proof.
	Observe that it follows from the kernel-range decomposition and left-invertibility of $\cf$ that $\Ker(\cf^*)^\perp=\Rng(\cf)$.
	Moreover, it is a direct consequence of \cite[Proposition~4.8(i)]{J-K-2019} that $f\in\Rng(\cf)$ if and only if $f$ is constant on preimages $\phi^{-1}(\{x\})$ for $x\in X$ and $f\circ\phi^{-1}\in L^2(\mu)$.
	It is known that $\cf\pp{1} f =hf$ for $f\in L^2(\mu)$ (see \cite{singh-1974}). 
	Note that, since $\cf$ is left-invertible, $\phi$ is an onto mapping.
	Indeed, suppose, to the contrary, that $y\in X$ is such that $y\notin\phi(X)$. 
	Then $\cf f=\cf (f + \chi_{y})$ for any $f\in L^2(\mu)$, where $\chi_y\in L^2(\mu)$ is the characteristic function of $\{y\}$.
	This contradicts injectivity of $\cf$.
	
	If $k=1$, then the result follows directly from \cite[Proposition~4.8(ii)]{J-K-2019}, so we limit our consideration to the case when $k\ge2$.
	We prove the required equivalence using an induction argument.
	Assume that $n\in\{2,\dots,k\}$ and the conclusion of the theorem holds for $i\in\{1,\dots,n-1\}$.
	Suppose $\cf$ satisfies the $n$-kernel condition.
	Since $\phi$ is surjective and $\h$ is constant on $\phi^{-i}(\{x\})$ for $i\in\{1,\dots,n-1\}$ for every $x\in X$, relation $\h \circ \phi^{-i}$ is a function for $i\in\{1,\dots,n-1\}$.
	This implies that
	\begin{align}
		\label{eq:composition-derivative}
		\h \circ \phi^{-i}\circ \phi^{i}= \h, \quad i\in\{1,\dots,n-1\}.
	\end{align}
	Observe that 
	\begin{align*}
		\cf\pp{n}f &= \cf^{*(n-1)} \Big( \h \cf^{n-1}f\Big) \\
		&= \cf^{*(n-1)} \Big(\h f\circ \phi^{n-1} \Big) \\ 
		&\overset{\eqref{eq:composition-derivative}}= \cf^{*(n-1)} \cf^{n-1}(\h \circ \phi^{-n+1} f) \\
		&\overset{\eqref{eq:composition-derivative}}= \h (\h \circ\phi^{-1})\dots(\h \circ\phi^{-n+1})f, \quad f\in L^2(\mu).
	\end{align*}
	Since  $\cf$ satisfies the $n$-kernel condition, $\cf\pp{n}f$ is constant on $\phi^{-1}(\{x\})$ for all $f\in\Rng(\cf)$ and for every $x\in X$. Hence $\h$ is constant on $\phi^{-n}(\{x\})$ for every $x\in X$.
	To prove the reverse implication it is enough to invert the argument used above.
	We leave the details to the reader.
	
	The rest of the proof follows directly from the induction argument.
\end{proof}

\section{$m$-isometric unilateral operator valued weighted shifts}

In this section we prove few results related to unilateral operator valued weighted shifts having positive and commuting weights, in particular, we characterize all $m$-isometric operators in the aforementioned class.

We begin with recalling required notation and definitions related to operators given by spectral integrals.
Let $X$ be a set and $\ascr$ be a $\sigma$-algebra of subsets of $X$.
Suppose that $E\colon\ascr\to\B$ is a spectral measure.
Consistently with the definition from Section 2, for a set $A\subseteq X$ we say that a property $\mathscr{W}$ of elements belonging to $A$ holds for $E$-almost every $x\in A$ (which is abbreviated to a.e.\ $[E]$ on $A$), if there exists $\Delta \in\ascr$ such that $E(\Delta)=0$ and $\mathscr{W}$ holds for every $x\in A\setminus\Delta$. 
Define $\hat\C=\C\cup\{\infty\}$ and  
\begin{align*}
	L^\infty(X, E) &= \{f\colon X\to\hat\C\:\: |\:\:f \text{ is measurable and there exists } c>0 \\ 
	&\text{ such that } |f| \le c \text{ a.e.\ }[E]\text{ on }X\}.
\end{align*}
For $f\in L^\infty(X, E)$ we set $\|f\|_\infty = \operatorname{inf}\{c>0 : |f| \le c \text{ a.e.\ }[E]\text{ on }X\}$.

The following lemma aggregates results that can be deduced from \cite[Propositions~4.17-4.18]{schmudgen-2012}.
We shall use it in the sequel.

\begin{lemma} 
	\label{lem:spectral-measure-zero-operator}
	Suppose that $X$ is a nonempty set, $\ascr$ is a $\sigma$-algebra of subsets of $X$, $E\colon\ascr\to\B$ is a spectral measure and $f\in L^\infty(X, E)$.
	Then
	\begin{enumerate}
		\item $\int_X f\mathrm{d}E\in\B$ and $\|f\|_\infty = \| \int_X f\mathrm{d}E \|$,
		\item if $f\ge0$ a.e.\ $[E]$ on $X$, then $\int_X f\mathrm{d}E$ is a positive operator,
		\item $\int_X f\mathrm{d}E = 0$ if and only if $f=0$ a.e.\ $[E]$ on $X$.
	\end{enumerate}	
\end{lemma}

Before moving to the next result we recall some notions from the measure theory.
Suppose that $(\mathcal{X},\ascr)$ and $(\mathcal{Y},\bscr)$ are measure spaces and $\phi\colon\mathcal{X}\to\mathcal{Y}$ is a bijection.
We say that $\phi$ is a \textit{measure isomorphism} (sometimes called a \textit{bi-measurable mapping}) if both $\phi$ and $\phi^{-1}$ are measurable. 
If $\mathcal{X}$ is a topological space, then by $\mathcal{B}(\mathcal{X})$ we denote the $\sigma$-algebra of Borel sets of $\mathcal{X}$.
From now and on $\operatorname{card}(X)$ stands for the cardinality of an arbitrary set $X$.

Although the following is well known to specialists, we provide its proof for the reader's convenience.

\begin{proposition}
	\label{pro:spectral-measure}
	Suppose $\{S_n\}_{n=0}^\infty\subseteq\B$ is a sequence of positive and commuting operators. 
	Then there exist a spectral measure $E\colon\mathcal{B}([0,1])\to\B$ and a sequence $\{\xi_n\}_{n=0}^\infty$ of $\mathcal{B}([0,1])$-measurable bounded functions $\xi_n\colon[0,1]\to\rbb_+$ such that $S_n=\int_{[0,1]} \xi_n \mathrm{d}E$ for all $n\in\zbb_+$.
\end{proposition}
\begin{proof}
	It is a direct consequence of the spectral theorem that for all $n\in\zbb_+$ there exists $\alpha_n\in(0,\infty)$, spectral measure $E_n\colon\mathcal{B}(I_n)\to\B$ such that $S_n=\int_{I_n} \lambda_n \mathrm{d}E_n$ for some $\mathcal{B}(I_n)$-measurable bounded function $\lambda_n\colon I_n\to\rbb_+$, where $I_n = [0,\alpha_n]$ (see \cite[Theorem~5.1]{schmudgen-2012} for details).
	Moreover, in view of \cite[Corollary 5.6]{schmudgen-2012}, since the operators $\{S_n\}_{n=0}^\infty$ are commuting with each other, so do their spectral measures $\{E_n\}_{n=0}^\infty$.
	As all the above measures are compactly supported, it follows from \cite[Proposition~4]{stochel-1987} that there exists a product spectral measure $F\colon\mathcal{B}(\mathcal{X})\to\B$, where $\mathcal{X}=\prod_{n=0}^{\infty} I_n$ is such that 
	\begin{align}
		\label{eq:spectral-measure-product}
		F\Big(\prod_{i\in I}\Delta_i \times \prod_{i\in\nbb\setminus I}I_i \Big) = \prod_{i\in I} E_i(\Delta_i),
		\quad \Delta_i\in\mathcal{B}(I_i) \text{ for } i\in I, \operatorname{card}(I)<\infty.
	\end{align}
	We infer from \eqref{eq:spectral-measure-product} that $E_n(\Delta_n) = F(\Delta_n\times\prod_{i\in\nbb\setminus \{n\}}I_i)$ for every $n\in\zbb_+$ and $\Delta_n\in\mathcal{B}(I_n)$.
	In view of this, \cite[Theorem~5.4.10]{B-S-1987} implies that 
	\begin{align}
		\label{eq:spectral-form-of-weights}
		S_n = \int_{I_n} \lambda_n \mathrm{d}E_n = \int_{\mathcal{X}} f_n(\lambda) \mathrm{d}F, \quad n\in\zbb_+,
	\end{align}
	where $\lambda=(\lambda_1, \lambda_2, \dots)$ and $f_n\colon\mathcal{X}\to I_n$ are such that $f_n(\lambda)=\lambda_n$ for all $n\in\zbb_+$ and $\lambda\in\mathcal{X}$.
	It follows from \cite[Proposition~8.1.3]{cohn-1980} that $\mathcal{X}$  with the product topology, as a countable product of complete separable metric spaces, is also a complete separable metric space.
	Thus, since $\operatorname{card}(\mathcal{X})=\operatorname{card}([0,1])$, it can be deduced from \cite[Theorem~2.12]{parthasarathy-1967} that there exists a measure isomorphism $\phi\colon\mathcal{X}\to[0,1]$.
	Observe that $f_n\circ\phi^{-1}$ is $\mathcal{B}([0,1])$-measurable for every $n\in\zbb_+$.
	Then, in view of this, \cite[Theorem~5.4.10]{B-S-1987} and \eqref{eq:spectral-form-of-weights} we see that
	\begin{align*}
		S_n = \int_{\mathcal{X}} (f_n \circ\phi^{-1}\circ\phi) \mathrm{d}F = \int_{[0,1]} (f_n\circ\phi^{-1}) \mathrm{d}E, \quad n\in\zbb_+,
	\end{align*}
	where $E\colon\mathcal{B}([0,1])\to\B$ is a spectral measure given by $E(\Delta) = F(\phi^{-1}(\Delta))$ for  $\Delta\in\mathcal{B}([0,1])$.
	This completes the proof.
\end{proof}

In what follows, a sequence of operators $\{S_n\}_{n=0}^\infty\subseteq\B$ is said to be \textit{uniformly bounded} (resp. \textit{uniformly bounded from below}) if there exists $M>0$ (resp. $c>0$) such that $\|S_n\|\le M$ for all $n\in\zbb_+$ (resp. $\|S_nx\|\ge c\|x\|$ for all $x\in\hh$ and $n\in\zbb_+$).

Denote by $\ell^2(\hh)$ the Hilbert space $\oplus_{n=0}^\infty \hh$.
Let us recall that by a \textit{unilateral operator valued weighted shift} we understand an operator $S\in\B[\ell^2(\hh)]$ associated with a sequence of uniformly bounded invertible operators $\{S_n\}_{n=0}^\infty\subseteq\B$ such that 
\begin{align*}
	Se_n(f) &= e_{n+1}(S_nf), \quad f\in\hh, n\in\zbb_+,
\end{align*}
where $e_n(f)$ is the element of $\ell^2(\hh)$ such that $f$ is on $n$-th position and all other positions contain the zero element of $\hh$. 
For simplicity, we denote this by $\ushift{S}$.
It can be easily verified that
\begin{align}
	\label{eq:shift-adjoint-basis}
	S^*e_n(f) = \begin{cases}
		 e_{n-1}(S_n^*f) & \text{if } n>0, \\
		 0 & \text{if } n=0,
	\end{cases} \quad f\in\hh.
\end{align}

Operator $S\in\B[\ell^2(\hh)]$ is called a \textit{diagonal operator} if there exists a uniformly bounded sequence of invertible operators $\{S_n\}_{n=0}^\infty\subseteq\B$ such that 
\begin{align*}
	Se_n(f) &= e_{n}(S_nf), \quad f\in\hh, n\in\zbb_+.
\end{align*}
If $S\in\B[\ell^2(\hh)]$ is a diagonal operator associated with the sequence $\{S_n\}_{n=0}^\infty$, then we say that the elements from the sequence $\{S_n\}_{n=0}^\infty$ are the \textit{elements located on the diagonal} of $S$.

Observe that, since $m$-isometric operators are bounded from below, it is easily verified that the sequence of weights of an $m$-isometric unilateral operator valued weighted shift is uniformly bounded from below (to prove this statement use the inequality from the definition for boundedness from below to $e_n(f)$ for every $n\in\zbb_+$ and $f\in\hh$).
Furthermore, note that, in light of \cite[Proposition 2.5(i)]{jablonski-2004}, $S$ is an $m$-isometry for $m\in\nbb$ if and only if 
\begin{align}
	\label{eq:m-iso-shift}
	0 = \|f\|^2 + \sum_{p=1}^m (-1)^p\binom{m}{p} \|S_{[p,s]}f\|^2, \quad s\in\zbb_+, f\in\hh,
\end{align} 
where $S_{[p,s]} = S_{p-1+s}\dots S_s$ for $p\in\nbb$.

Assume that $m\in\zbb_+$ and $E\colon\mathcal{B}([0,1])\to\B$ is a spectral measure.
A mapping $W\colon\zbb_+\times[0,1]\to(0,\infty)$ is called an \textit{$E$-measurable family of polynomials of degree at most $m$} if the following assertions hold
\begin{enumerate}
	\item $W(0,x)=1$ a.e.\ $[E]$,
	\item $W(n,x)$ is a polynomial in $n$ of degree at most $m$ for $E$-almost every $x\in[0,1]$,
	\item the functions $W(n,x)$ and $\sqrt{\frac{W(n+1, x)}{W(n,x)}}$ are $\mathcal{B}([0,1])$-measurable for every $n\in\zbb_+$.
\end{enumerate}

Observe that in the definition stated above the assumption that $\sqrt{\frac{W(n+1, x)}{W(n,x)}}$ are $\mathcal{B}([0,1])$-measurable for every $n\in\zbb_+$ is superfluous. 
This particular assumption follows directly from the others and basic properties of measurable functions.
Hence, we do not need to verify it while checking properties of an $E$-measurable family of polynomials.

The following theorem is the main result of this section and yields a characterization of $m$-isometric unilateral operator valued weighted shifts with positive and commuting weights for $m\in\nbb$.

\begin{theorem}
	\label{thm:unilateral-m-iso}
	Suppose $m\in\nbb$ and $\{S_n\}_{n=0}^\infty\subseteq\B$ is a uniformly bounded sequence of invertible, positive and commuting operators. 
	Then the following are equivalent{\em :}
	\begin{enumerate}
		\item $\ushift{S}$ in an $m$-isometry,
		\item there exist a spectral measure $E\colon\mathcal{B}([0,1])\to\B$ and an $E$-measurable family $W$ of polynomials of degree at most $m-1$ such that 
		\begin{align*}
			S_n = \int_{[0,1]} \sqrt{\frac{W(n+1, x)}{W(n,x)}} E(dx), \quad n\in\zbb_+.
		\end{align*}
	\end{enumerate}
\end{theorem}
\begin{proof}
	First, observe that Proposition \ref{pro:spectral-measure} implies that there exist a spectral measure $E\colon\mathcal{B}([0,1])\to\B$ and a sequence of $\mathcal{B}([0,1])$-measurable functions $\{\xi_n\}_{n=0}^\infty$ such that 
	\begin{align*}
		S_n = \int_{[0,1]} \xi_n \mathrm{d}E, \quad n\in\zbb_+.
	\end{align*}
	Moreover, since operators $\{S_n\}_{n=0}^\infty$ are invertible, by \cite[Proposition~4.19]{schmudgen-2012} we can choose functions $\{\xi_n\}_{n=0}^\infty$ so that their images are entirely contained within $(0,\infty)$.
	Thus, in view of \eqref{eq:m-iso-shift} and \cite[Proposition~4.1]{schmudgen-2012} we see that $\ushift{S}$ is an $m$-isometry if and only if 
	\begin{align*}
		0 &= \|f\|^2 + \sum_{p=1}^m (-1)^p\binom{m}{p} \|S_{[p,s]}f\|^2 \\
		&= \|f\|^2 + \sum_{p=1}^m (-1)^p\binom{m}{p} \int_{[0,1]} |\xi_{p-1+s}^2(x)\dots\xi_s^2(x)| \langle E(dx) f, f\rangle \\
		&= \|f\|^2 + \int_{[0,1]} \sum_{p=1}^m (-1)^p\binom{m}{p}  \xi_{p-1+s}^2(x)\dots\xi_s^2(x) \langle E(dx) f, f\rangle, \quad s\in\zbb_+, f\in\hh \\
	\end{align*}
	which, by Lemma \ref{lem:spectral-measure-zero-operator}(iii), is equivalent to 
	\begin{align}
		\label{eq:spectral-m-iso}
		0 = 1+ \sum_{p=1}^m  (-1)^p\binom{m}{p}  \xi_{p-1+s}^2\dots\xi_s^2, \quad s\in\zbb_+, \text{ a.e.\ } [E].
	\end{align} 
	
	(i)$\Rightarrow$(ii).
	Since $\ushift{S}$ is an $m$-isometry, \eqref{eq:spectral-m-iso} holds.
	We define an $E$-measurable family $W$ in the following way 
	\begin{align*}
		W(n,x) = \begin{cases}
			1 & \text{if } n=0,\\
			\xi_{n-1}^2(x)\dots\xi_0^2(x) & \text{if } n\in\nbb,
		\end{cases} \quad x\in[0,1].
	\end{align*}
	Observe that \eqref{eq:spectral-m-iso} is equivalent to 
	\begin{align*}
		0 = \sum_{p=0}^m  (-1)^p\binom{m}{p} \frac{W(p+s, x)}{W(s,x)}, \quad s\in\zbb_+, \text{ a.e.\ } [E].
	\end{align*}
	and, by \cite[Proposition~2.1]{J-B-S-2020}, the latter implies that $W(n,x)$ is a polynomial in $n$ of degree at most $m-1$ for $E$-almost every $x\in[0,1]$.
	
	(ii)$\Rightarrow$(i).
	For the proof of this part it is enough to note that \eqref{eq:spectral-m-iso} holds using again \cite[Proposition~2.1]{J-B-S-2020}.
	Hence $\ushift{S}$ is an $m$-isometry.
\end{proof}

\begin{remark} \noindent 
	\begin{enumerate}
		\item Observe that in Proposition \ref{pro:spectral-measure} instead of $[0,1]$ one can choose an arbitrary compact interval in $\rbb_+$ and modify the statement of Theorem \ref{thm:unilateral-m-iso} accordingly.
		\item It is easy to see that Theorem \ref{thm:unilateral-m-iso} generalizes \cite[Theorem 2.1]{A-L-2016} which characterizes $m$-isometric classical unilateral weighted shifts.
	\end{enumerate}
\end{remark}

Let us recall some notation from \cite{J-B-S-2020}. 
Define an operator $\triangle\colon\rbb^{\zbb_+}\to\rbb^{\zbb_+}$ so that $\triangle\gamma=\gamma'$ for $\gamma=\{\gamma_n\}_{n=0}^\infty\in\rbb^{\zbb_+}$ where $\gamma'_n = \gamma_{n+1} - \gamma_n$ for $n\in\zbb_+$.
It is easily seen that $\triangle$ is linear on $\rbb^{\zbb_+}$.
For $n\in\nbb$ and $k\in\zbb_+$ set 
\begin{align*}
	(n)_k = \begin{cases}
		1 & \text{ if } k=0,\\
		\prod_{i=0}^{k-1} (n-i) & \text{ otherwise.}
	\end{cases}
\end{align*}
Note that $(n)_k$ is a polynomial in $n$ of degree $k$ for any $k\in\zbb_+$.

Suppose $\gamma = (\gamma_0, \gamma_1, \dots)$ is a sequence of functions such that $\gamma_n\colon[0,1]\to\rbb_+$ for every $n\in\zbb_+$. 
We define
\begin{align*}
	\gamma(x) = (\gamma_0(x), \gamma_1(x), \dots), \quad x\in[0,1].
\end{align*}

Let $\mathcal{C}$ be a class of operators.
Following the definition from \cite[Section 3]{J-K-2019} we say that the problem of determining whether or not there exists an extension of a given initial finite sequence of positive values to a sequence of weights for some unilateral weighted shift so that the shift is in class $\mathcal{C}$ is called the \textit{completion problem} for operators in class $\mathcal{C}$. 
We generalize this problem to unilateral weighted shifts with operator weights.

The following result is a consequence of Theorem \ref{thm:unilateral-m-iso} and provides a solution to the completion problem for $m$-isometric unilateral operator valued weighted shifts with $m-1$ initial weights which are positive and commuting (cf.\ \cite[Theorem~3.3]{jablonski-2004}).

\begin{corollary}
	\label{cor:completion-problem}
	Suppose $\{S_n\}_{n=0}^{m-2}\subseteq\B$ are positive, invertible, commuting operators for some $m\ge2$ and $E\colon\mathcal{B}([0,1])\to\B$ is a spectral measure such that $S_n=\int_{[0,1]}\xi_n\mathrm{d}E$ for $n\in\{0,\dots,m-2\}$.
	Assume that there exist $C$, $c\in(0,\infty)$ such that 
	\begin{align}
		\label{eq:shifts-completion-boundedness}
		C> \frac{\sum_{k=0}^{m-1}\frac{(\triangle^k \gamma)_0}{k!} (n+1)_k}{\sum_{k=0}^{m-1}\frac{(\triangle^k \gamma)_0}{k!} (n)_k} \ge c, \quad  n\ge m, \text{ a.e.\ } [E], 
	\end{align}
	and 
	\begin{align}
		\label{eq:shifts-completion-boundedness-2}
		C> \frac{\sum_{k=0}^{m-1}\frac{(\triangle^k \gamma)_0}{k!} (m)_k}{\gamma_{m-1}} \ge c, \quad \text{ a.e.\ } [E], 
	\end{align}
	where 
	$\gamma=(1,\xi_0^2,\dots,\prod_{k=0}^{m-2}\xi_k^2,0,\dots)$ is a sequence of functions.
	Then there exists a sequence $\{S_n\}_{n=m-1}^\infty\subseteq\B$ of operators such that $\ushift{S}$ is an $m$-isometric unilateral operator valued weighted shift with positive and commuting weights.
\end{corollary}
\begin{proof}
	Observe that $\gamma$ is a sequence such that $\gamma_n\colon[0,1]\to(0,\infty)$ is a measurable function for all $n\in\{0,\dots,m-1\}$.
	Let us define $W\colon\zbb_+\times[0,1]\to(0,\infty)$ such that 
	\begin{align}
		\label{eq:family-completion-problem}
		W(n,x) = \begin{cases}
			\gamma_n(x) & \text{ for } n\in\{0,\dots,m-1\}, \\
			\sum_{k=0}^{m-1}\frac{(\triangle^k \gamma(x))_0}{k!} (n)_k & \text{ for } n\ge m,
		\end{cases} \quad x\in[0,1].
	\end{align}
	Now observe that $W$ is an $E$-measurable family of polynomials of degree at most $m-1$ due to \cite[formula (2.2)]{J-B-S-2020} and \cite[Proposition~2.1]{J-B-S-2020}.
	Using the definition of $\gamma$ and \eqref{eq:family-completion-problem} we can easily verify that
	\begin{align*}
		S_n &= \int_{[0,1]} \xi_n \mathrm{d}E = \int_{[0,1]} \sqrt{\frac{\gamma_{n+1}(x)}{\gamma_n(x)}} E(dx) \\ 
		&\overset{\eqref{eq:family-completion-problem}}= \int_{[0,1]} \sqrt{\frac{W(n+1, x)}{W(n,x)}} E(dx), \quad n\in\{0,\dots,m-2\}.
	\end{align*}
	Define $S_n=\int_{[0,1]} \sqrt{\frac{W(n+1, x)}{W(n,x)}} E(dx)$ for $n\ge m-1$.
	In view of \eqref{eq:shifts-completion-boundedness}, \eqref{eq:shifts-completion-boundedness-2}, Lemma~\ref{lem:spectral-measure-zero-operator} and \cite[Proposition~4.19]{schmudgen-2012}, sequence $\{S_n\}_{n=0}^\infty$ is a uniformly bounded sequence of positive, invertible operators. 
	Theorem \ref{thm:unilateral-m-iso} completes the proof.
\end{proof}

The next example shows a useful application Corollary \ref{cor:completion-problem} and establishes a very simple sufficient condition for existence of a solution to the completion problem in class of $3$-isometric unilateral operator valued weighted shifts with positive and commuting weights for two initial weights.

\begin{example}
	Suppose that $S_0$, $S_1\in\B$ are positive, invertible and commuting operators.
	It follows from Proposition \ref{pro:spectral-measure} that there exist a spectral measure $E\colon\mathcal{B}([0,1])\to\B$ and $\mathcal{B}([0,1])$-measurable functions $\xi_0$, $\xi_1\colon[0,1]\to(0,\infty)$ such that $S_n= \int_{[0,1]} \xi_n \mathrm{d}E$ for $n\in\{0,1\}$.
	Suppose that there exist $C$, $c\in(0,\infty)$ such that
	\begin{align*}
		C > 1 + 
		\frac{\xi_0^2 -1  + n(\xi_0^2\xi_1^2 - 2\xi_0^2 + 1)}
			{1 + n(\xi_0^2 -1) + \frac{n(n-1)}{2}(\xi_0^2\xi_1^2 - 2\xi_0^2 + 1)} 
		\ge c, \quad n\ge m, \text{ a.e.\ } [E]
	\end{align*}
	and
	\begin{align*}
		C >  
		\frac{1 + 3\xi_0^2\xi_1^2 - 3\xi_0^2}{\xi_0^2\xi_1^2}
		\ge c, \quad \text{ a.e.\ } [E]
	\end{align*}
	
	This implies that \eqref{eq:shifts-completion-boundedness} and \eqref{eq:shifts-completion-boundedness-2} hold with $\gamma$ given as in Corollary \ref{cor:completion-problem} .
	Thus, it follows from Corollary \ref{cor:completion-problem} that there exists a sequence of positive, invertible and commuting operators $\{S_n\}_{n=2}^\infty$ such that $\ushift{S}$ is a $3$-isometry and $\{S_n\}_{n=0}^\infty$ is a commuting family.
\end{example}

Finally, let us observe that the solution to the $3$-isometric completion problem established in the above example is actually unique.
This is a direct consequence of an analogical uniqueness argument to the one used in \cite[Theorem~3.3]{jablonski-2004} for $2$-isometries.
We generalize this remark in the result below.

\begin{proposition}
	If $m\ge2$ and $\{S_n\}_{n=0}^{m-2}\subseteq\B$ consists of positive, invertible and commuting operators, then, if exists, the solution to the completion problem for $m$-isometric unilateral operator valued weighted shifts for $\{S_n\}_{n=0}^{m-2}$ is unique in the class of unilateral shifts with positive and commuting weights.
\end{proposition}
\begin{proof}
	Suppose $\{S_n\}_{n=m-1}^\infty$ and $\{S_n'\}_{n=m-1}^\infty$ extend $\{S_n\}_{n=0}^{m-2}$ to $m$-isometric unilateral operator valued weighted shifts with positive and commuting weights.
	It follows directly from \eqref{eq:m-iso-shift} that
	\begin{align}
		\label{eq:m-iso-uniqueness}
		\notag \| S_{m-1}\dots S_0f\|^2 &= (-1)^m \Big( -\|f\|^2 - \sum_{p=1}^{m-1} (-1)^p \binom{m}{p} \|S_{p-1}\dots S_0f\|^2\Big) \\
		&= \| S_{m-1}'S_{m-2}\dots S_0f\|^2, \quad f\in\hh.
	\end{align}
	Combined with the fact that the weights are invertible, positive and commuting this implies that $S_{m-1}=S_{m-1}'$.
	Using \eqref{eq:m-iso-uniqueness} with an analogical induction argument proves that $S_n=S_n'$ for every $n\ge m-1$. 
	The verification of the last statement is left to the reader. 
	This concludes the proof.
\end{proof}

\section{The Wold-type decomposition for $m$-isometries}

Following \cite[Definition 1.1]{shimorin-2001} we say that operator $T\in\B$ admits the \textit{Wold-type decomposition}, if the following two hold
\begin{enumerate}
	\item $\RngInf(T)$ reduces $T$ to a unitary operator,
	\item $\bigvee\{T^n(\Ker(T^*))\}_{n=0}^\infty\oplus \RngInf(T) = \hh$.
\end{enumerate}
Observe that in the class of analytic operators condition  (i) from the above definition is trivially satisfied.
Moreover, we say that $T\in\B$ has the \textit{wandering subspace property}, if $\bigvee\{T^n(\Ker(T^*))\}_{n=0}^\infty=\hh$.
The last definition was introduced in \cite[Definition~2.4]{shimorin-2001}.

Before presenting the next example let us recall the definition of composition operators on directed graphs with one circuit from \cite{J-K-2019}.
Let $\kappa\in\nbb$, $\eta\in\nbb\cup\{\infty\}$.
Suppose
\begin{align*}
	X = \{x_1, \ldots, x_\kappa\} \cup \bigcup_{i=1}^\eta \Big\{x_{i,j}\colon j \in\nbb \Big\},
\end{align*}
where $\{x_i\}_{i=1}^{\kappa}$ and $\{x_{i,j}\}_{i=1}^{\eta}{_{j=1}^\infty}$ are disjoint sets of distinct elements of $X$.
Assume that $(X,\ascr,\mu)$ is a discrete measure space.
Let $\phi$ be a self-map of $X$ such that 
\begin{align}
	\label{eq:phi-definition}
	\phi(x) = \begin{cases}
		x_{i,j-1} & \text{ if } x = x_{i,j} \text{ for some } 1\le i\le\eta \text{ and } j\in\nbb\setminus \{1\},\\
		x_\kappa & \text{ if } x = x_{i,1} \text{ for some } 1\le i\le\eta \text{ or } x=x_1, \\
		x_{i-1} & \text{ if } x = x_i \text{ for some } i\in \{j\in\nbb\colon 2\le j \le  \kappa\}.
	\end{cases}
\end{align}
An operator $\cf$ defined on $L^2(\mu)$ with $\phi$ defined above is called a \textit{composition operator on a directed graph with one circuit}.
For details and known results regarding this class of operators the reader is referred to \cite{J-K-2019}.

The following example which shows that \cite[Theorem 2.5]{A-C-J-S-2017} does not hold for analytic $3$-isometries (cf. \cite[Example 3.1]{A-C-T-2018}).
	
\begin{example}
	Suppose $\cf\in\B[L^2(\mu)]$ is a composition operator on a directed graph with one circuit with $\kappa=\eta=1$.  
	It follows from \cite[Corollary~2.12]{J-K-2019} that $\cf$ is a strict $3$-isometry if and only if there exists a monomial $w(x) = ax+b$ for some $a>0$ and $b\in\rbb_+$ such that $\mu(x_{1,i})=w(i)$ for $i\in\nbb$.
    Assume that $\cf$ is a strict $3$-isometry such that $b>0$ and $\cf$ satisfies the kernel condition.
    It follows from Theorem \ref{thm:composition-kernel-condition} and \eqref{eq:phi-definition} that $\h$ is constant on $x_1$ and $x_{1,1}$, which combined with \cite[(2.10)]{J-K-2019} implies that 
    \begin{align}
    	\label{eq:measure-kernel-condition}
    	\mu(x_1)=\frac{(a+b)^2}{a} . 
    \end{align}
	Note that it is a direct consequence of the above that 
	\begin{align*}
		\Ker(\cf^*) &= \{f\in L^2(\mu) : f(x_{1,1}) = -f(x_1)\Big(1+\frac{b}{a}\Big) \text{ and } f(x_{1,n}) = 0 \text{ for } n\ge 2\}.
	\end{align*}
	For $k\in\zbb_+$ set $M_k=\cf^k(\Ker(\cf^*))$ and observe that
	\begin{align}
		\label{eq:example-spaces}
		M_k = &\{f\in L^2(\mu) : f(x_1)=\dots=f(x_{1,k}), f(x_{1,k+1}) = -f(x_1)\Big(1+\frac{b}{a}\Big) \notag \\ 
		& \text{ and } f(x_{1,n}) = 0 \text{ for } n > k+1\}.
	\end{align}
	It can be now verified that $M_2$ is not orthogonal to $M_3$.
	Indeed, let $f\in M_2$, $g\in M_3$ be such that $f(x_1)=g(x_1)=1$.
	Then
	\begin{align*}
		\langle f,g \rangle 
		&\overset{\eqref{eq:example-spaces}}= f(x_1)g(x_1)\mu(x_1) + \sum_{n=1}^\infty f(x_{1,n})g(x_{1,n})\mu(x_{1,n}) \\
		&= \mu(x_1) + \mu(x_{1,1}) + \mu(x_{1,2}) - \mu(x_{1,3})\Big(1+\frac{b}{a}\Big) \\
		&\overset{\eqref{eq:measure-kernel-condition}}=  \frac{(a+b)^2}{a} + (a+b) + (2a+b) - \Big(1+\frac{b}{a}\Big)(3a+b)
		= a \ne 0.
	\end{align*}
	Hence condition (iii) from \cite[Theorem 2.5]{A-C-J-S-2017} does not hold.
	Furthermore, it follows from \cite[Corollary 4.3]{J-K-2019} that $\cf$ is analytic.
\end{example}

Let us also note that in the above example $\cf$ is expansive (i.e. $\cf^*\cf\ge I$).
This can be inferred from \cite[Lemma 2.3]{jablonski-2003} and \cite[(2.10)]{J-K-2019}.

The following theorem establishes an equivalent condition for an analytic $m$-isometry to admit the Wold-type decomposition with additional property regarding orthogonality of images of the kernel of the adjoint of the operator under its powers for $m\ge2$.

\begin{theorem}
	\label{thm:analytic-m-iso}
	Suppose $T\in\B$ is an analytic $m$-isometry for some $m\ge2$.
	Then the following are equivalent{\em:}
	\begin{enumerate}
		\item $T$ satisfies the $(m-1)$-kernel condition,
		\item spaces $\{T^n(\Ker(T^*))\}_{n=0}^\infty$ are orthogonal to each other and $T$ admits the Wold-type decomposition.
	\end{enumerate}
\end{theorem}
\begin{proof}	
	If $m=2$, then the theorem follows directly from \cite[Theorem 2.5]{A-C-J-S-2017}.
	Therefore, assume that $m\ge 3$ and define $M_n=T^n(\Ker(T^*))$ for $n\in\zbb_+$.
	
	(i)$\Rightarrow$(ii).  
	We begin with showing that spaces $\{M_n\}_{n=0}^\infty$ are orthogonal to each other.
	Observe that, thanks to Lemma \ref{lem:kernel-condition} and (i), $T$ satisfies the $k$-kernel condition for every $k\in\nbb$.	
	Suppose that $p\in\zbb_+$ and $k\ge p + 1$. 
	Then 
	\begin{align*}
		\langle T^{p}f, T^k g\rangle &= \langle T^*T\pp{p}f, T^{k-p-1}g\rangle =  0, \quad f,g\in\Ker(T^*),
	\end{align*}
	as $T\pp{p}f\in\Ker(T^*)$ due to the $p$-kernel condition.
	Hence $\{M_n\}_{n=0}^\infty$ are orthogonal to each other.
	
	Due to the fact that $T$ is analytic, in order to finish the proof it is sufficient to show that closed linear span of $\{M_n\}_{n=0}^\infty$ coincides with $\hh$.
	Before showing this let us prove that
	\begin{align}
		\label{eq:ort-sum}
		\Ker(T^{*n}) = \bigoplus_{k=0}^{n-1} M_k, \quad n\in\nbb.
	\end{align}
	Indeed, suppose that \eqref{eq:ort-sum} holds for some $n\ge2$. 
	Let $f\in\Ker(T^{*(n+1)})$. 
	Since, the range of $T^n$ is closed, it follows from the kernel-range decomposition that $f=T^{n}h + g$ for some $h\in\hh$ and $g\in\Ker(T^{*n})$.
	Thus
	\begin{align*}
		0 = T^{*(n+1)} f = T^{*} T\pp{n} h + T^{*(n+1)}g = T^{*} T\pp{n}h,
	\end{align*}
	and, hence, by the $n$-kernel condition and due to the fact that $T\pp{n}$ is invertible (see Remark~\ref{rem:invertiblity}) we see that $h\in\Ker(T^*)$.
	Moreover, if $x\in M_n$ and  $z\in\Ker(T^{*n})$, then there exists $y\in\Ker(T^*)$ such that
	\begin{align*}
		\langle x, z\rangle = \langle T^n y, z\rangle = \langle y, T^{*n}z\rangle = 0.
	\end{align*}
	Hence $\Ker(T^{*(n+1)})\subseteq M_n\oplus \Ker(T^{*n})$. 
	In view of this and Proposition \ref{pro:analytic-inclusions} we see that \eqref{eq:ort-sum} holds for $n+1$.
	To complete the induction argument it is enough to repeat a similar argument as above to prove \eqref{eq:ort-sum} holds for $n=2$.
	
	Now, define $A=\bigcup_{n=1}^\infty \Ker(T^{*n})$ and observe that due to the left-invertibility of $T$ we obtain
	\begin{align*}
		A^\perp &= \Big\{ f\in\hh \colon 0=\langle f,h\rangle \text{ for all } h\in \bigcup_{n=1}^\infty \Ker(T^{*n})\Big\} \\
		&= \{ f\in\hh \colon f\in\overline{\Rng(T^n)}=\Rng(T^n) \text{ for all }n\in\nbb\} \\
		&= \RngInf(T).
	\end{align*}
	This combined with \eqref{eq:ort-sum}, \cite[Corollary 2.10]{conway} and $T$ being analytic  implies that closed linear span of $\{M_n\}_{n=0}^\infty$ coincides with $\hh$.
	Thus $T$ admits the Wold-type decomposition.	

	(ii)$\Rightarrow$(i). 
	Let $k\in\nbb$.
	Since $T$ is analytic and admits the Wold-type decomposition, the closed linear span of $\{M_n\}_{n=0}^\infty$ coincides with $\hh$.
	This and the orthogonality of subspaces $\{M_n\}_{n=0}^\infty$ imply that for every $h\in\hh$ there exists a sequence $\{f_n\}_{n=0}^\infty\subseteq\Ker(T^*)$ such that
	\begin{align}
		\label{eq:element-decomposition}
		Th = \sum_{n=0}^\infty T^n f_n = \sum_{n=1}^\infty T^n f_n,
	\end{align}
	where the last equality can be deduced from the fact that $f_0=0$ as $Th$ is orthogonal to $\Ker(T^*)$ due to the kernel-range decomposition.
	
	Observe that for $h\in\hh$ and $f\in\Ker(T^*)$ we have
	\begin{align*}
		\langle T\pp{k} f, Th\rangle 
		&\overset{\eqref{eq:element-decomposition}}= \langle T\pp{k} f, \sum_{n=1}^\infty T^nf_n\rangle \\
		&=  \sum_{n=1}^\infty \langle T\pp{k} f,  T^{n}f_n\rangle 
		=  \sum_{n=1}^\infty \langle T^k f,  T^{n+k}f_n\rangle 
		= 0,
	\end{align*}
	where the last equality follows from the fact that $\{M_n\}_{n=0}^\infty$ are orthogonal to each other.
\end{proof}

The following corollary provides an affirmative answer to \cite[Question 1.2]{A-C-T-2018} in some subclass of operators by establishing a sufficient condition for an expansive $m$-isometric operator to admit the Wold-type decomposition.

\begin{corollary}
	\label{cor:wold-type-answer}
	An expansive $m$-isometric operator which satisfies the $(m-1)$-kernel condition for $m\ge2$ admits the Wold-type decomposition.
\end{corollary}
\begin{proof}
	Since for $m=2$ this result is well known (see \cite{shimorin-2001} and \cite[Theorem~2.5]{A-C-J-S-2017}), we limit our proof to the case $m\ge3$.
	Suppose that $T\in\B$ is an expansive $m$-isometry.
    Recall that it follows from \cite[Proposition 3.4]{shimorin-2001} that $\RngInf(T)$ is reducing for $T$ and $T_{|\RngInf(T)}$ is unitary.
    
    Assume that $T$ satisfies the $(m-1)$-kernel condition.
	Observe that the $(m-1)$-kernel condition and $m$-isometricity are invariant under taking a restriction to a reducing subspace (by Lemma \ref{lem:kernel-condition-reducing}  and \cite[page 388]{A-S-1995}, respectively).
	Hence the analytic part of $T$ has the wandering subspace property by Theorem \ref{thm:analytic-m-iso} and therefore $T$ admits the Wold-type decomposition.
\end{proof}

The following result is a generalization of \cite[Theorem 2.5]{A-C-J-S-2017} (cf.\ \cite[Theorem 4.1]{olofsson-2004}) for $m$-isometries for $m\ge2$ and it is proved using similar techniques. 

\begin{theorem}
	\label{thm:analytic-m-iso-shifts}
	Suppose that $T\in\B$ is an analytic $m$-isometry for some $m\ge2$.
	Then the following are equivalent{\em:}
	\begin{enumerate}
		\item $T$ satisfies the $(m-1)$-kernel condition,
		\item there exists a sequence of unitary isomorphisms $\{V_n\}_{n=0}^\infty$ such that 
			\begin{enumerate}
				\item $V_n\colon T^n\Ker(T^*)\to \Ker(T^*)$ for every $n\in\zbb_+$,
				\item $T$ is unitarily equivalent to a unilateral operator valued weighted shift $\ushift{S}$, where $S_n=V_{n+1} T_{|T^n\Ker(T^*)} V_n^{-1}$ for every $n\in\zbb_+$.
			\end{enumerate}
	\end{enumerate}
\end{theorem}
\begin{proof}
	Note that for $m=2$ this result is already proved (see \cite[Theorem 2.5]{A-C-J-S-2017}).
	Therefore, we limit our considerations to the case $m\ge3$.
	For the rest of the proof define $M_n=T^n(\Ker(T^*))$ for all $n\in\zbb_+$.
	
	(i)$\Rightarrow$(ii). 
	It follows from Theorem \ref{thm:analytic-m-iso} that $T$ admits the Wold-type decomposition and the subspaces $\{M_n\}_{n=0}^\infty$ are orthogonal to each other.
	Since for every $n\in\zbb_+$ the operator $T_{|M_n}\colon M_n\to M_{n+1}$ is a linear bijection, then the spaces $\{M_n\}_{n=0}^\infty$ are unitarily equivalent to each other (see \cite[Problem 42]{halmos-1974}).
	Let $V_n\colon M_n\to M_0$ be the above unitary isomorphism for all $n\in\zbb_+$.
	It is now a matter of repeating the same argument as in \cite[Theorem 2.5]{A-C-J-S-2017} to show that $VT=SV$, where $V\colon\hh\to\ell^2(M_0)$ is a unitary diagonal operator with operators $\{V_n\}_{n=0}^\infty$ located on its diagonal and $S$ is the unilateral operator valued weighted shift with weights $\{V_{n+1} T_{|M_n} V_n^{-1}\}_{n=0}^\infty\subseteq\B[M_0]$.
	This completes the proof of this implication.
	
	(ii)$\Rightarrow$(i). 
	Let $k\in\{1,\dots,m-1\}$. 
	It is a matter of performing an easy computation to verify that for $f\in\Ker(T^*)$ we have
	\begin{align*}
		S\pp{k}e_0(f) \overset{\eqref{eq:shift-adjoint-basis}}= e_0(S_0^*\dots S_{k-1}^*S_{k-1}\dots S_0 f) = e_0(T\pp{k}f),
	\end{align*}
	which implies that $T$ satisfies the $k$-kernel condition.
	Hence, the proof is completed.
\end{proof}

Observe that in the class of isometries condition (ii) from Theorem \ref{thm:analytic-m-iso} and condition (ii) from Theorem \ref{thm:analytic-m-iso-shifts} are trivially satisfied. 
Those are well known facts and follow directly from the proof of the Wold decomposition (see \cite[Chapter I, Theorem 1.1]{N-F-2010}).

The characterization from Theorem \ref{thm:analytic-m-iso-shifts} can be used as a tool to determine whether an analytic $m$-isometry is unitarily equivalent to a unilateral operator valued weighted shift.
For example Theorem \ref{thm:analytic-m-iso-shifts} combined with \cite[Theorem~4.9]{J-K-2019} implies that the class of analytic $m$-isometric composition operators on a directed graph with one circuit is essentially different than the class of $m$-isometric unilateral operator valued weighted shifts as there are no $2$-isometric composition operators on a directed graph with one circuit that satisfies the kernel condition.
The following result generalizes this remark to the class of $m$-isometric composition operators on a directed graph with one circuit having one element.

\begin{corollary}
	\label{cor:composition-unitary-equivalence}
	Suppose that $\cf\in\B[L^2(\mu)]$ is a composition operator on a directed graph with one circuit with $\kappa=1$, $\eta\in\nbb\cup\{\infty\}$ and $\cf$ is an $m$-isometry for $m\ge2$.
	Then $\cf$ is not unitarily equivalent to any unilateral operator valued weighted shift.
\end{corollary}
\begin{proof}
	Since $\cf$ is an $m$-isometry, it can be deduced from \cite[Corollary~2.12]{J-K-2019} that for all $i\le\eta$ there exists a real polynomial $w_i$ of degree at most $m-2$ such that 
	\begin{align}
		\label{eq:measure-m-iso-composition}
		\mu(x_{i,j}) = w_i(j),  \quad j\in\nbb.
	\end{align}
	Assume that $\cf$ is unitarily equivalent to unilateral operator valued weighted shift. 
	Since, by \cite[Corollary~4.3]{J-K-2019}, operator $\cf$ is analytic, it follows from Theorem \ref{thm:analytic-m-iso-shifts} that $\cf$ satisfies the $(m-1)$-kernel condition.
	Hence, by Lemma \ref{lem:kernel-condition}, $\cf$ satisfies the $k$-kernel condition for every $k\in\nbb$.
	This combined with \eqref{eq:measure-m-iso-composition}, Theorem \ref{thm:composition-kernel-condition} and \cite[(2.10)]{J-K-2019} implies that 
	\begin{align}
		\label{eq:measure-m-iso-composition-2}
		1+\frac{1}{\mu(x_1)} \sum_{k=1}^\eta w_k(1)= \frac{w_i(j+1)}{w_i(j)}, \quad i\le\eta, j\in\nbb.
	\end{align}

	Set $a=1+\frac{1}{\mu(x_1)} \sum_{k=1}^\eta w_k(1)$ and observe that $a>1$.
	It follows from \eqref{eq:measure-m-iso-composition} and \eqref{eq:measure-m-iso-composition-2} that 
	$w_n(j+1) = a^jw_n(1)$ for $j\in\nbb$. 
	This contradicts the fact that $w_i$ is a polynomial for every $i\le\eta$ and completes the proof.
\end{proof}

Finally, we give two simple examples of applications of Theorem \ref{thm:analytic-m-iso-shifts} to show that certain analytic $m$-isometric operators are unitarily equivalent to unilateral operator valued weighted shifts.

\begin{example}
	Assume $\hh=\C$.
	Let $m\ge 2$ and $s$, $t$ be polynomials in real variable of degree exactly $m-1$ having positive values on $\zbb_+$.  
	Set $\ushift{S}$ and $\ushift{T}$, where $S_n=\sqrt{\frac{s(n+1)}{s(n)}}I$ and $T_n=\sqrt{\frac{t(n+1)}{t(n)}}I$ for $n\in\zbb_+$. 
	It is a direct consequence of Theorem~\ref{thm:unilateral-m-iso} that both $S$ and $T$ are $m$-isometries.
	Moreover, each one of them satisfies the $k$-kernel condition for every $k\in\nbb$ and, hence, due to Lemma \ref{lem:kernel-condition-reducing} so does $A:=S\oplus T$.
	Now, observe that Theorem \ref{thm:analytic-m-iso-shifts} implies that $A$ is unitarily equivalent to operator valued weighted shift.
\end{example}

\begin{example}
	Let $T\in\B[\ell^2]$ be given in the following way
	\begin{align*}
		T(x_0, x_1, \dots) = \Big(\frac{s_0}{2}(x_0 +ix_1), \frac{s_0}{2}(ix_0 -x_1), \frac{s_1}{\sqrt{2}}(x_0 -ix_1), s_2x_2, s_3x_3,\dots  \Big), 
	\end{align*}
	where $(x_0,x_1,\dots)\in\ell^2$ and $s_n=\sqrt{\frac{s(n+1)}{s(n)}}$ for some real polynomial $s$ of degree at most $2$ having positive values on $\zbb_+$.
	It is a matter of performing routine calculations to verify that 
	\begin{align*}
		T^*(x_0, x_1, \dots) = \Big( \frac{s_0}{2}(x_0 -ix_1) + \frac{s_1}{\sqrt{2}}x_2, \frac{s_0}{2}(-ix_0-x_1) +\frac{s_1}{\sqrt{2}}ix_2 , s_2x_3,\dots    \Big)
	\end{align*}
	where $(x_0,x_1,\dots)\in\ell^2$ and 
	\begin{align*}
		\Ker(T^*)= \{(x_0,x_1,\dots)\in\ell^2 : x_0=ix_1 \text{ and } x_n=0 \text{ for } n\ge2\}.
	\end{align*}
	One can check that $T$ is an analytic $3$-isometry by checking \eqref{eq:m-iso-condition} holds on the orthonormal basis and by computing $\RngInf(T)$.
	Since $T$ satisfies the $2$-kernel condition, it follows from Theorem \ref{thm:analytic-m-iso-shifts} that $T$ is unitarily equivalent to a unilateral operator valued weighted shift.
	Moreover, it can be easily verified that weights of the shift can be defined on $\mathbb{C}$ since $\Ker(T^*)$ is a one dimensional Hilbert space. 
\end{example}

\section*{Acknowledgments}

The author would like to thank Zenon Jan Jab\l o\'nski for helpful discussions concerning the subject of the paper.

The author would like to express their gratitude to Professor Jan Stochel for providing the idea of the proof of Proposition \ref{pro:spectral-measure} which significantly improves the presentation of Theorem \ref{thm:unilateral-m-iso}.

The author would also like to thank Sameer Chavan and Md.\ Ramiz Reza for several useful comments and anonymous reviewers for careful reading of the manuscript and suggestions that helped to improve the final version of the paper.


\begin{thebibliography}{9}
	
	\bibitem{A-L-2016}
		B. Abdullah, T. Le,
		\emph{The structure of $m$-isometric weighted shift operators},
		Oper. Matrices, \textbf{10} (2016) 2, 319--334.
		
	\bibitem{A-S-1995}
		J. Agler, M. Stankus,
		\emph{$m$-isometric transformations of Hilbert space. I},
		Integr. Equ. Oper. Theory {\bf{21}} (1995), 383--429.
	
	\bibitem{A-C-T-2018}
		A. Anand, S. Chavan, S. Trivedi,
		\emph{Analytic $m$-isometries without the wandering subspace property},
		Proc. Amer. Math. Soc. \textbf{148} (2020) 5, 2129--2142.

	\bibitem{A-C-J-S-2017}
		A. Anand, S. Chavan, Z. J. Jab\l o\'nski, J. Stochel,
		\emph{A solution to the Cauchy dual subnormality problem for 2-isometries}, 
		J. Funct. Anal. \textbf{277} (2019) 12, 108292.
			
	\bibitem{B-S-1987}
		M. S. Birman, M. Z. Solomjak,
		\emph{Spectral Theory of Self-Adjoint Operators in Hilbert space},
		D. Reidel Publishing, Co., Dordrecht, 1987.
		
	\bibitem{cohn-1980}
		D. L. Cohn,
		\emph{Measure Theory},
		Birkh\"auser Boston, 1980.
	
	\bibitem{conway}
		J. B. Conway,
		\textit{A Course in Functional Analysis}, 
		Springer-Verlag, New York, 1990.
		
	\bibitem{halmos-1974}	
		P. R. Halmos,
		\emph{A Hilbert Space Problem Book},
		Springer-Verlag, New York, 1974.
					
	\bibitem{halmos-1961}	
		P. R. Halmos,
		\emph{Shifts on Hilbert spaces},
		J. Reine Angew. Math.  {\bf 208} (1961), 102--112.

	\bibitem{jablonski-2003}
		Z. J. Jab\l o\'nski,
		\emph{Hyperexpansive composition operators},
		Math. Proc. Camb. Phil. Soc. \textbf{135} (2003), 513--526.
		
	\bibitem{jablonski-2004}
		Z. J. Jab\l o\'nski,
		\emph{Hyperexpansive operator valued unilateral weighted shifts},
		Glasgow Math. J. \textbf{46} (2004), 405--416.

	\bibitem{J-B-S-2020}
		Z. J. Jab\l o\'nski, I. B. Jung, J. Stochel,
		\emph{{$m$-Isometric operators and their local properties}},
		Linear Algebra Appl. \textbf{596} (2020), 49--70.

	\bibitem{J-K-2019}
		Z. J. Jab\l o\'nski, J. Ko\'smider,
		\emph{{$m$-isometric composition operators on directed graphs with one circuit}},
		Integr. Equ. Oper. Theory, (accepted).
		
	\bibitem{nordgren-1978}
		E. Nordgren,
		\emph{Composition operators on Hilbert spaces},
		Lecture Notes in Math. \textbf{693} (1978), 37--63.
		
	\bibitem{olofsson-2004}
		A. Olofsson,
		\emph{A von Neumann-Wold decomposition of two-isometries},
		Acta Sci. Math. Szeged, \textbf{70} (2004), 715--726.
		
	\bibitem{parthasarathy-1967}
		K. R. Parthasarathy,
		\emph{Probability measures on metric spaces},
		Academic Press New York, 1967,
			
	\bibitem{richter-1988}
		S. Richter,
		\emph{Invariant subspaces of the Dirichlet shift},
		J. Reine Angew. Math. {\bf 386} (1988), 205--220.
		
	\bibitem{schmudgen-2012}
		K. Schm\"udgen,
		\emph{Unbounded Self-adjoint Operators on Hilbert space},
		Graduate Texts in Mathematics, 265, Springer, Dordrecht, 2012.
			
	\bibitem{shimorin-2001}
		S. Shimorin,
		\emph{Wold-type decompositions and wandering subspaces for operators close to isometries},
		J. Reine Angew. Math. {\bf 531} (2001), 147--189.
			
	\bibitem{singh-1974}
		R. K. Singh,
		\emph{Compact and quasinormal composition operators},
		Proc. Amer. Math. Soc. \textbf{45} (1974) 1, 80--82.
		
	\bibitem{stochel-1987}
	 	J. Stochel,
	 	\emph{The Fubini Theorem for Semi-spectral Integrals and Semi-spectral Representations of Some Families of Operators},
	 	Univ. Iagiel, Acta Math. \textbf{26} (1987), 17--27.
	 	
	\bibitem{N-F-2010}
		B. Sz.-Nagy, C. Foias, H. Bercovici, L. K\'erchy,
		\emph{Harmonic Analysis of Operators on Hilbert space},
		Springer-Verlag, New York, 2010.

\end{thebibliography}
\end{document}